\documentclass[11pt, reqno]{amsart}
\usepackage[lmargin=2.5cm,tmargin=3cm,bmargin=3cm,rmargin=2.5cm]{geometry} 
\usepackage[colorlinks=true,urlcolor=blue, citecolor=red,linkcolor=blue,linktocpage,pdfpagelabels, bookmarksnumbered,bookmarksopen]{hyperref}
\usepackage[table]{xcolor}
\usepackage[]{color}
\usepackage{tabularx}
\usepackage{multirow}
\usepackage{amsmath, amsfonts, amssymb, amscd}
\usepackage{mathrsfs}
\usepackage{mathabx}
\usepackage{eucal}
\usepackage{tikz}
\usepackage{graphicx, graphics, epsfig}
\usepackage{subfig}
\usepackage{epstopdf}
\usepackage{float}
\usepackage{setspace}
\usepackage{enumerate}
\usepackage{marginnote}
\usepackage{dsfont}

\newtheorem{theorem}{Theorem}[section]
\newtheorem{proposition}[theorem]{Proposition}
\newtheorem{corollary}[theorem]{Corollary}
\newtheorem{lemma}[theorem]{Lemma}
\newtheorem{remark}[theorem]{Remark}

\newcommand{\p}{\partial}

\newcommand{\cA}{{\mathcal A}}
\newcommand{\cB}{{\mathcal B}}

\newcommand{\cD}{{\mathcal D}}

\newcommand{\cG}{{\mathcal G}}
\newcommand{\cJ}{{\mathcal J}}
\newcommand{\R}{{\mathbb R}}
\newcommand{\Z}{{\mathbb Z}}

\newcommand{\T}{{\mathbb T}}
\newcommand{\N}{{\mathbb N}}

\DeclareMathOperator{\sech}{sech}

\newtheorem*{TA}{Theorem A}

\begin{document}

\title[ Non-Isotropic Schr\"odinger ]{Well-posedeness for the non-isotropic Schr\"odinger equations on cylinders and periodic domains }
\author[A.J. Corcho]{ Ad\'an J. Corcho}
\address{Department of Mathematics, University of C\'ordoba, C\'ordoba, Spain.}
\email{a.corcho@uco.es}
\thanks{}
\author{Marcelo Nogueira}
\address{Universidade Federal de Ouro Preto, Campus Mariana - ICSA, 35420-057, Mariana, MG, Brazil.}
\email{marcelo.nogueira@ufop.edu.br}
\thanks{}
\author{Mahendra Panthee}
\address{Department of Mathematics, University of Campinas, 13083-859, Campinas, SP,  Brazil.}
\email{mpanthee@ime.unicamp.br}
\thanks{}



\begin{abstract}
The initial value problem (IVP) for the non-isotropic Schr\"odinger   equation posed on the two-dimensional cylinders and $\T^2$ is considered. The IVP is shown to be locally well-posed for small initial data in  $H^s(\T\times\R)$ if $s\geq0$. For the IVP posed on $\R\times\T$, given data are  considered in the anisotropic Sobolev spaces thereby obtaining the local well-posedness result in $H^{s_1, s_2}(\R\times\T)$, if $s_1\geq0$ and $s_2>\frac12$. In the purely periodic case, a particular case of the IVP is shown to be locally well-posed for any given initial data in $H^s(\T^2)$ if $s>\frac14$.  In some cases, ill-posedness issues are also considered showing that the IVP posed on $\T\times \R$, in the focusing case, is ill-posed in the sense that the application data-solution fails to be uniformly continuous for data in $H^s(\T\times\R)$ if $-\frac12\leq s<0$.
\end{abstract}

\maketitle

\setcounter{equation}{0}
\section{Introduction}
In this work, we study the initial value problem (IVP) for the non-isotropic Schr\"odinger (NI-NLS) equation posed on two-dimensional domains
$\cD=\T\times \R$,   $\cD=\R\times \T$  or $\cD=\T\times \T$ with $\T=\R/2\pi\Z$. More precisely, we consider 

\begin{equation}\label{model}
\begin{cases}
i \partial_t u +  \varepsilon\partial_x^2u + \partial_y^2u+\alpha  \partial_x^{4}u = \pm |u|^2 u, & t\in \R,\; (x, y) \in \cD,\medskip \\
u(0, x, y) = \phi(x,y),   &(x, y) \in \cD,
\end{cases}
\end{equation}
where $u$ is a complex valued function, $\varepsilon \in \{0, 1\}$ and  $\alpha \in \R\backslash\{0\}$.

The  unitary group associated with the free evolution of the NI-NLS equation in \eqref{model} is given by 
\begin{equation}\label{group-1}
S_{\alpha, \varepsilon}(t) := e^{i t (\varepsilon\partial_{x}^2 + \partial_{y}^{2}  + \alpha  \partial_{x}^{4}  ) }. 
\end{equation}
By using the Fourier variables $(k_1, k_2)$ belonging to $\Z\times \R$, $\R\times \Z$ or $\Z^2$, the unitary group $S_{\alpha, \varepsilon}(t)$ is defined  by
\begin{equation}\label{group-2}
\big[S_{\alpha, \varepsilon}(t)f\big]^{\wedge}(k_1, k_2)=e^{-it(\omega_{\alpha, \varepsilon}(k_1) +k_2^2)}\widehat{f}(k_1, k_2), 
\end{equation}
with
\begin{equation}\label{group-3}
\omega_{\alpha, \varepsilon}(k_1)= \varepsilon k_1^2 -\alpha k_1^4, 
\end{equation}
where 
\begin{equation}\label{FourierTransf}
\widehat{f}(k_1,k_2)=C_{\cD}\iint_{\cD}f(x,y)e^{-i(xk_1 + yk_2)}dxdy
\end{equation}
with a suitable constant $C_{\cD}$. In what follows the Fourier variables $(k_1, k_2)$ will be renamed by $(n,\xi)$ when $\cD=\T\times \R$, $(\xi, n)$ when $\cD=\R\times \T$, and $(n_1, n_2)$ when $\cD=\T^2$.

\medskip 
The NI-NLS equation \eqref{model} with $\varepsilon=1$  is widely used in fiber arrays and was introduced in \cite{FP-1, FP-2} while studying the effect of small perturbations on critical self-focusing by reducing the perturbed critical NLS equation  to a simpler system of modulation equations.  This sort of analysis is motivated by the physical application because the faster propagation in the optical fibers may be attained using an array of coupled optical waveguides arranged on a line in which the pulses undergo 2d self-focusing \cite{FP-2}. This model has also been used in other physical situations, for example, to describe the propagation of solitons in fiber arrays \cite{AAT} and of ultra-short laser pulses in a planar waveguide medium with anomalous time dispersion \cite{WF2002}. For detailed physical motivation and other practical applications, we refer to \cite{ FIS, FP-1, FP-2}. There are several works on the fourth order NLS equations, for example \cite{HN2021, NP2015, W2006} are just a few to mention. 

\medskip 
Sufficiently regular solutions of the  NI-NLS  equation \eqref{model} for $\varepsilon=1$ enjoy  the following conservation laws, viz., the mass
\begin{equation}\label{Mass}
\mathcal{M}(t) := \iint_{D}|u(t,x,y)|^2dxdy =  \mathcal{M}(0)     
\end{equation}
and energy
\begin{equation}\label{Energy}
 \mathcal{E}(t) :=  \iint_{D}\Big( |\nabla u(t,x,y)|^2 -  \alpha |\p^2_xu(t,x,y)|^2  \pm \tfrac{1}{4} |u(t,x,y)|^4\Big)dxdy = \mathcal{E}(0).    
\end{equation}

The well-posedness issues for the IVP \eqref{model} for $\varepsilon=1$ posed on the continuous domain $\R^n$ are widely considered in the literature, see for example \cite{GC2008, GC2009, SG2019} and references therein. In particular, the authors in \cite{GC2008} derived dispersive estimates exploiting the time decay property  of the associated group and proved the local well-posedness result for given data in $H^s(\R^2)$, $s\geq 0$. Also, using the $L^2$ conserved quantity \eqref{Mass}
$$\|u( t, \cdot )\|_{L^2(\R^2)}= \|\phi\|_{L^2(\R^2)},$$
satisfied by the flow of \eqref{model} they obtained the global well-posedness result for the $L^2$-data. For the problems posed on higher dimension and with general nonlinearities we refer to \cite{SG2019} and references therein.

\medskip 
Recently, the study of the IVPs posed on the product spaces, like cylinders $\R\times \T$ has attracted much attention of several authors, see for example 
\cite{BCP2021,   HTT2014, TT2001, TTV2014} and references therein. Particularly, we mention the result in \cite{TT2001} where the authors proved that the IVP \eqref{model} with $\alpha=0$ and $\varepsilon=1$ is globally well-posed for small data in $L^2(\R\times\T)$
or  $L^2(\T\times\R)$. This result coincides with the one for data in $L^2(\R^2)$ \cite{CW1990, Tsu1987} improving the purely periodic case where the well-posedness result holds for data in $H^s(\T\times\T)$, $s>0$ \cite{Bo1993, Bo1999}. Motivated from this result, it is natural to ask whether one can obtain a better well-posedness results for the IVP \eqref{model} posed on domains $\T\times \R$, $\R\times \T$ or $\T^2$ when $\alpha \neq 0$. We dedicate this work to respond this question.

\medskip
Before establishing the main results, we record some notations that will be used throughout this work.

\begin{itemize}
	\item $|J|$  denotes the Lebesgue measure of a set $J\subset \R$,
	\item for any $A, B >0$,\;   $A\lesssim B$ means that there exists a positive constant $c$ such that $A\leq cB$,
	\item given two real numbers $a$ and $b$ we set $a \vee b:= \min\{a, b\}$ and $a\wedge b:= \max\{a, b\}$,
	\item $\mathfrak{m}(\cdot)$ denotes the product measure of the one-dimensional Lebesgue measure of a measurable set $J\subset \R$ and counting measure,
	\item $\lfloor x \rfloor$  denotes the integer part of $x$.
\end{itemize}

\section{Statement of the main results}\label{sec-2}

\setcounter{equation}{0}

In this section, we present the main results of this work. For the sake of clarity, we separate the results for the IVP posed on $\T\times\R$, $\R\times\T$ and $\T^2$ in each subsection.

\subsection{Well-posedness on the cylindrical domain $\T\times\R$} Our first result concerns obtaining an $L^4\!\!-\!\!L^2$ Strichartz estimate on cylinder $\T\times \R$ for the group $S_{\alpha, \varepsilon}(t)$, similar to 
the one obtained for $\alpha=0$ in \cite{TT2001} for the group $S(t) := e^{i t (\partial_{x}^2 + \partial_{y}^{2})}.$
More precisely, considering  $\alpha<0$ we will prove the following result.

\begin{proposition}[Strichartz estimate on $I\times\T\times\R$]\label{Th-Strichartz}
	Let $\alpha<0$ be a fixed real number and $I\subset\R_t$ an interval containing $t=0$. Then, there exists a positive constant $C_I$, depending only on the length of $I$, such that 
	\begin{equation}\label{Strichartz}
	\|S_{\alpha, \varepsilon}(t)\phi\|_{L^4(I\times\T\times \R)}\le C_I \|\phi\|_{L^2(\T\times\R)},
	\end{equation}
	for any $\phi \in L^2(\T\times\R)$. Moreover, there exists a positive constant $\widetilde{C}_I$, depending only on the length of $I$, such that 
	\begin{equation}\label{NH-Strichartz}
		\left\|\int_0^tS_{\alpha, \varepsilon}(t-t')f(t',\cdot)dt'\right\|_{L^4(I\times\T\times \R)}\leq \widetilde{C}_I\left\| f\right\| _{L^{4/3}(I\times\T\times\R)},
	\end{equation}
	for any $f\in L^{4/3}(I\times\T\times\R)$.
\end{proposition}

As in \cite{TT2001}, in the context of IVP  for the cubic elliptic NLS, Proposition \ref{Th-Strichartz} combined with Picard iteration scheme
applied to the integral equation
\begin{equation}\label{int-equation}
	u(t)= S_{\alpha, \varepsilon}(t)\phi \mp i\int_0^t S_{\alpha, \varepsilon}(t-t')|u(t')|^2u(t')dt'
\end{equation}		
imply the following result:
\begin{theorem}[Well-posedness in $L^2$]\label{Th-WP-L2}
Let $\alpha <0$. Then the  IVP \eqref{model} is globally well-posed for sufficiently small initial data $\phi \in L^2(\T\times\R)$.
\end{theorem}

\begin{remark}\label{Rem-2.3}
 As we will see, our approach fails in the case  $\alpha >0$ due to the bad algebraic structure of the symbol in the periodic direction in the sense that  
the polynomial equation $\varepsilon x^2 -\alpha x^4 =\gamma$ possesses real solutions for any positive number $\gamma$  only in the case $\alpha \le 0$.
\end{remark}

\begin{theorem}[Ill-posedness below $L^2$]\label{Th-IP-below-L2}
Let $\alpha <0$ and consider the focusing regime $(-|u|^2u)$ in \eqref{model}. Then the IVP \eqref{model} is ill-posed in the sense that the mapping data-solution is not uniformly continuous 
on bounded sets of initial data in $H^s(\T\times \R)$ whenever $-\frac12\leq s<0$.
\end{theorem}

\begin{remark}
  If we consider the  IVP \eqref{model} with initial data $\phi$ depending only on the $y$-variable and $\phi\in H^s(\R)$, it follows that 
	$\tilde{\phi}(x, y):=\phi(y)\in  H^s(\T\times\R)$ with
	$$\|\tilde{\phi}\|_{H^s(\T\times\R)}=\|\phi\|_{H^s(\R)},$$
	and the solutions of the IVP 
	\begin{equation}\label{Ex-2}
		\begin{cases}
			i\p_tw+ \p^2_yw=\pm|w|^2w,\qquad & y\in \mathbb{R},\;\; t\in \R,\medskip \\
			w(y,0)=\phi(y),
		\end{cases}
	\end{equation}
	are also solutions of \eqref{model}. Assuming the existence of local solutions, the IVP \eqref{Ex-2} is  ill-posed for $s\le -\frac12$ in both cases (focusing and defocusing) in the following sense:
	for any $\delta>0$  the  flow-map
	$$
		\Phi: u_0\in H^s(\R) \longmapsto w\in C([0,\delta];  H^s(\R))
    $$
    is discontinuous everywhere in $H^s(\R)$ (norm inflation argument, see \cite{Oh}). In some sense, this remark complements Theorem \ref{Th-IP-below-L2}.
\end{remark}

\subsection{Well-posedness on the cylindrical domain $\R\times\T$}
For the IVP \eqref{model} posed on the cylindrical domain $\R\times\T$, the structure of the symbol involved created an obstacle to find Strichartz estimate of the form given in Proposition  \ref{Th-Strichartz}. However, we exploited the  one dimensional Strichartz estimate in the $x$-variable, proved in \cite{GC2009}, and used it to get a new Strichartz-type estimate involving the periodic $y$-variable as well, see Proposition \ref{PropI} below. Using this new Strichartz-type estimate we obtain some local well-posedness results for given initial data in $H^s(\R\times\T)$.

\medskip 
In this case, first we consider the given initial data in the anisotropic Sobolev spaces  $\mathcal{H}_{x,y}^{s_1,s_2}(\R\times\T)$ defined as the completion of $C_0^{\infty} (\mathbb{R} \times \mathbb{T})$ with respect to the norm 
\begin{equation}\label{Anisot-Spaces}
\|f\|_{\mathcal{H}_{x,y}^{s_1,s_2}} = \|(1 - \partial_x^2)^{s_1/2}(1 - \partial_y^2)^{s_2/2} f \|_{L^{2}(\mathbb{R}\times \mathbb{T})} = \|J_x^{s_1}J_y^{s_2} f \|_{L^{2}(\mathbb{R} \times \mathbb{T})}. 
\end{equation}
In particular, for $(s_1, s_2) = (0, s)$, we have 
\[
\|f\|_{\mathcal{H}_{x,y}^{0, s}} = \|J^{s}_y f \|_{L^2_{x,y}}.
\]
In this setting, we prove the following local result.

\begin{theorem}\label{Th-LWP-aniso}
Let $\alpha \neq 0$. Then the  IVP \eqref{model} is locally well-posed for  any given initial data $\phi \in \mathcal{H}_{x,y}^{0,s}(\R\times\T)$ whenever $s>\frac12$. 
\end{theorem}

\subsection{Well-posedness on the periodic domain  $\T^2$}

For the IVP \eqref{model} posed on the purely periodic domain $\T^2$, we use decoupling theory developed by Bourgain-Demeter \cite{BD2015} to get a new Strichartz estimate. More precisely, we use the following result from \cite{BD2015}.

\begin{TA}\label{The-Decoup}
Let $S$ be a compact $C^2$ hypersurface in $\R^n$ with positive definite second fundamental form. Let $\Lambda\subset S$ be a $\delta^{\frac12}$-separated set, and let $R\gtrsim \delta^{-1}$. Then for each $\epsilon>0$,
\begin{equation} \label{eq-decoup}
\Big(\frac1{|B_R|}\int_{B_R}\Big|\sum_{\xi\in\Lambda}a_{\xi}e^{2\pi i\, \xi\cdot x}\Big|^p\Big)^{\frac1p}\lesssim_{\epsilon} \delta^{\frac{n+1}{2p}-\frac{n-1}{4}-\epsilon}\|a_{\xi}\|_{L^2},
\end{equation}
if $p\geq \frac{2(n+1)}{n-1}$.

\end{TA}

It is worth emphasizing that in the application of Theorem A, homogeneity in each variable of the linear part plays a crucial role in obtaining a compact hypersurface with a positive definite second fundamental form. In our case, the lack of homogeneity in the $x$-variable creates an additional obstacle in obtaining such a surface. This fact compelled us to consider a particular case of the IVP~\eqref{model} with $\varepsilon =0$ so as to ensure homogeneity in both $x$ and $y$ variables separately even though  it is non-isotropic. With this consideration, using  Theorem A we  obtain the following $L^4-L^2$-Strichartz estimate with $\frac18+\epsilon$ derivative loss.

\begin{proposition}\label{Prop-St2}
Let $\alpha <0$ and  $S_{\alpha,0}(t)$ be the linear group associated to the IVP \eqref{model} posed on 
$\T^2$. Then, there exists a positive constant $C$, such that 
	\begin{equation}\label{Strichartz-2}
	\|D_{x,y}^{-\frac18-\epsilon}S_{\alpha,0}(t)\phi\|_{L^4(I\times\T^2)}\le C \|\phi\|_{L^2(\T^2)},
	\end{equation}
	for any $\phi \in L^2(\T^2)$.
\end{proposition}

Furthermore, by a standard argument (see, for instance, \cite{BGT-05}), Proposition~\ref{Prop-St2} yields the following local well-posedness result.

\begin{theorem}\label{Th-LWP-T}
Let $\alpha <0$ and $\varepsilon =0$. Then the IVP \eqref{model} is locally well-posed for any given data in $H^s(\T^2)$ whenever $s>\frac14$.
\end{theorem}

\begin{remark}
The coefficient $\alpha \in \R$ is a modeling parameter. For technical reasons, in certain situations we consider $\alpha < 0$. This assumption mainly arises from the expression \eqref{group-3}, which is related to the phase of the semigroup in \eqref{group-2}. For instance, as mentioned in Remark~\ref{Rem-2.3}, our approach fails for $\alpha > 0$ when deriving the Strichartz estimate, which plays a crucial role in proving Theorem~\ref{Th-WP-L2} (the well-posedness in $L^2$). As can be seen in the proof of the Strichartz estimate, the restriction $\alpha < 0$ is necessary to ensure the existence of real solutions  to the polynomial equation
\[
\varepsilon x^2 - \alpha x^4 = \gamma,
\]
for $\gamma > 0$.
Moreover, the condition $\alpha < 0$ is required in Lemma~\ref{lemma-CS} to guarantee that the Gaussian curvature is positive, which is essential for applying the Decoupling Theorem in our setting. It would be very interesting if one could remove this restriction on $\alpha$.

From a physical perspective, the condition $\alpha < 0$ appears naturally as part of the nonparaxial correction to the NLS equation (see Section~2.1 in \cite{FIP}). Furthermore, as noted in \cite{FIS}, under the condition $\alpha < 0$, the second- and fourth-order dispersion terms act together, and one can also employ an anisotropic Gagliardo–Nirenberg inequality to obtain an {\em a priori} estimate of the solution for asymptotic analysis. See also \cite{GC2008} for similar considerations. From this point of view, the restriction we impose appears to be physically natural.
\end{remark}

\subsection{Structure of the paper} The paper is organized as follows. Section \ref{Section-TxR} is devoted to the proof of Strichartz estimates in Proposition \ref{Th-Strichartz}  and the ill-posedness results in Theorem \ref{Th-IP-below-L2}. In section \ref{Section-RxT} we develop a theory on domain $\R\times \T$, that lead to the proof of  Theorem \ref{Th-LWP-aniso}.  In section \ref{Section-TxT}, the decoupling theory developed by Bourgain-Demeter \cite{BD2015} is used to prove the Theorem 
\ref{Th-LWP-T}. Finally, Section \ref{sec-6} is devoted to present some concluding remarks.

\setcounter{equation}{0}
\section{Proof of results for cylindrical domain  $\T\times \R$}\label{sec-3}\label{Section-TxR}

In this section we consider the IVP \eqref{model} posed on the cylindrical domain  $\T\times \R$ and  provide proofs of the Strichartz estimates stated in Proposition \ref{Th-Strichartz} that would lead to the proof of the global well-posedness result for small data and the ill-posedness result stated in Theorem \ref{Th-IP-below-L2}. 

\subsection{Proof of Proposition \ref{Th-Strichartz}}
As in \cite{TT2001}, to get the Strichartz estimates in Proposition \ref{Th-Strichartz} it suffices to prove the following bilinear estimate. For the sake of clarity in exposition, we provide details considering $\varepsilon =1$. All the calculations for $\varepsilon =1$ also hold for case $\varepsilon =0$, the latter case being much simpler.
\begin{lemma}\label{Lem-bilinear-estimate}
	Let $\alpha<0$  be a fixed real number and consider  $u_1, u_2$ 
	 two functions defined on $\R \times \T \times \R$ with the following property 
	\[
	\text{supp}\,(\widehat{u_j}) \subset  \mathcal{E}_{j} := \big\{ (\tau, n, \xi) : \tfrac{1}{2} K_{j} \leq |\tau + \omega_{\alpha, 1}(n) + \xi^2| \leq 2 K_{j} \big\} 
	\]
	for $K_{j}$ ($j=1,2$). Then the following estimate holds
	\begin{equation}
		\|u_1 u_2\|_{L^2_{t,x,y}} \lesssim (K_1K_2)^{\frac12}\| u_1 \|_{L^2_{t,x,y}} \|u_2\|_{L^2_{t,x,y}}.\label{Lem-bilinear-estimate-a}
	\end{equation}
\end{lemma}

\begin{proof}
We set 
\begin{equation}\label{omega-first-condition}
\omega(n):= \omega_{\alpha,1}(n)= n^2 -\alpha n^{4},\quad \alpha<0.
\end{equation}

From  Cauchy-Schwarz inequality and Plancherel's identity  it is follows that 
\begin{equation}\label{proof-Lem-bilinear-estimate-a-01}
\|u_1 u_2\|_{L^2_{t,x,y}} \lesssim \Big(\sup\limits_{\tau, n , \xi} \frak{m}(\cA_{\tau, n , \xi})\Big)^{1/2}\|u_1\|_{L^2_{t,x,y}} \|u_2\|_{L^2_{t,x,y}},
\end{equation}
where
\[
\cA_{\tau, n, \xi} := \Big\{ (\tau_1, n_1, \xi_1): \tfrac{1}{2}K_1 \leq |\tau_1 + \omega(n_1) + \xi_1^2| \leq 2 K_1,\; \tfrac{1}{2}K_2 \leq |\tau - \tau_1  + \omega(n - n_1) + (\xi - \xi_1)^2| \leq 2K_2 \Big\}. 
\]
Hence, to get \eqref{Lem-bilinear-estimate-a} it is enough to estimate the mea\-sure of the set $\cA_{\tau, \xi, n}$.

Notice that 
\begin{enumerate}
\item[$\bold{(i)}$] If $(\tau_1, n_1, \xi_1) \in \cA_{\tau, n, \xi}$, then  
$\tau_1 \in J_1 \cap J_2$ (both $J_1$ and $J_2$ are intervals) such that 
$$|J_1 \cap J_2| \leq 4 (K_1\! \vee \! K_2).$$

\item[$\bold{(ii)}$] If  $(\tau_1, n_1, \xi_1) \in \cA_{\tau, n, \xi}$, eliminating $\tau_1$ (via triangular inequality) we obtain 
\begin{equation*}
\begin{split}
\big|(\xi_1 - \tfrac{\xi}{2})^2 & + \tfrac{1}{2} \omega(n_1) + \tfrac{1}{2}\omega(n - n_1) + \tfrac{\xi^2}{4} + \tfrac{\tau}{2} \big|=\\
&=\tfrac{1}{2} |\tau_1 + \omega(n_1) + \xi_1^2 + \tau - \tau_1 + \omega(n - n_1) + (\xi - \xi_1)^2| \leq K_1 + K_2. 
\end{split}
\end{equation*}

\item[$\bold{(iii)}$]   From $\bold{(ii)}$, if $(\tau_1, n_1, \xi_1) \in \cA_{\tau, n, \xi}$ then we have $(n_1, \xi_1) \in \cB_{\tau, n, \xi}$, where 
\[
\cB_{\tau, n, \xi}:= \big\{ (n_1, \xi_1) : \big|(\xi_1 - \tfrac{\xi}{2})^2  + \tfrac{1}{2} \omega(n_1) + \tfrac{1}{2} \omega(n - n_1) + \tfrac{\xi^2}{4} + \tfrac{\tau}{2} \big| \leq K_1 + K_2 \big\}. 
\]
\end{enumerate}
Hence, from $\bold{(i)}$ and $ \bold{(iii)}$, it is follows that 
\begin{equation}\label{proof-Lem-bilinear-estimate-a-02}
 \frak{m}(A_{\tau, \xi, n}) \lesssim (K_1\! \vee \! K_2) \frak{m}(B_{\tau, n, \xi}).
\end{equation}
In what follows we will prove that 
\begin{equation}\label{proof-Lem-bilinear-estimate-a-03}
\frak{m}(B_{\tau, n, \xi})\lesssim K_1\wedge K_2,
\end{equation}
which combined with \eqref{proof-Lem-bilinear-estimate-a-02} and \eqref{proof-Lem-bilinear-estimate-a-01} gives us the claimed estimate \eqref{Lem-bilinear-estimate-a}. 
\end{proof}

Finally, to prove \eqref{proof-Lem-bilinear-estimate-a-03}  it is sufficient to prove the following crucial lemma. 
\begin{lemma}\label{crucial-lem}
    Let $C \geq 0$ and $K \geq 1$ be constants, and 
    \[
    \mathcal{G}_{K} :=  \big\{ (n, \xi): C \leq \xi^2 + \tfrac{1}{2} \omega(n) + \tfrac{1}{2} \omega(n - n_0) \leq C + K \big\}, 
    \]
    with $n_0\in \N$. Then
    \[
    \mathfrak{m}( \mathcal{G}_{K}) \lesssim K,
    \]
    independently of $C$ and $n_0 \in \N$. 
\end{lemma}

\begin{proof} Since $\omega$ is an increasing even function we have that 
	\begin{equation}\label{proof-crucial-lem-00}
	\mathcal{G}_{K} \subset \mathcal{G}_{K,1} \cup \cG_{K,2},
	\end{equation}
	with 
	\begin{enumerate}
		\item[] $ \cG_{K,1} := \{ (\xi, n) : C- \omega(n)  \le \xi^2 \leq C + K -\omega(n-n_0) \},$

		\item[] $ \cG_{K,2}:= \{ (\xi, n) :C- \omega(n-n_0)  \le \xi^2 \leq C + K -\omega(n) \}.$
	\end{enumerate}

To justify \eqref{proof-crucial-lem-00} we observe that, since $\omega$ in \eqref{omega-first-condition} is an increasing function, $\cG_{K,1}$ 
contains the points of  $\cG_K$  with $|n-n_0|\le |n|$. Indeed, in that case we have 
$$\omega(n-n_0) \le  \tfrac{1}{2} \omega(n) + \tfrac{1}{2} \omega(n - n_0) \le \omega(n).$$
Therefore, 
$$C \leq \xi^2 + \tfrac{1}{2} \omega(n) + \tfrac{1}{2} \omega(n - n_0) \leq C + K$$ implies that 
$$C- \omega(n)  \le \xi^2 \leq C + K -\omega(n-n_0).$$
Similar argument shows that $\cG_{K,2}$ contains the points of  $\cG_K$ such that  $|n-n_0| >|n|$.
   
Before estimating the measures of the sets  $\mathcal{G}_{K,i} (i=1,2)$ we observe that the solution of the polynomial inequality $\omega(x):= x^2 -\alpha x^4 \le \gamma$, with $\alpha < 0$ and $\gamma>0$, 
is given by 
\begin{equation}\label{biquadratic-sol}
|x| \le x^*(\gamma):=\frac{1}{\sqrt{-2\alpha}} \sqrt{ \sqrt{ 1 - 4\alpha\gamma} - 1},
\end{equation}
with $\omega(x^*(\gamma))=\gamma$. In view of this fact we have that
\begin{equation}\label{proof-crucial-lem-01}
\begin{split}
\frak{m}(\cG_{K,1}) &= 2\sum_{|n -n_0|=0}^{\lfloor x^*(C+K)\rfloor}\sqrt{C+K - \omega(n-n_0)} - 2\sum_{|n|=0}^{\lfloor x^*(C)\rfloor}\sqrt{C- \omega(n)}\\
&=2\sum_{|n|=0}^{\lfloor x^*(C+K)\rfloor}\sqrt{C+K - \omega(n)} - 2\sum_{|n|=0}^{\lfloor x^*(C)\rfloor}\sqrt{C- \omega(n)}\\
&= :2\big( \Sigma_1(C,K) +  \Sigma_2(C,K) \big),
\end{split}
\end{equation}
where
\begin{align}
&\Sigma_1(C,K) := \sum_{|n|=0}^{\lfloor x^*(C)\rfloor}\Big(\sqrt{C+K - \omega(n)}  -\sqrt{C- \omega(n)}\Big)\\
&\Sigma_2(C,K) := \sum_{|n|=\lfloor x^*(C)\rfloor + 1}^{\lfloor x^*(C+K)\rfloor}\sqrt{C+K - \omega(n)}.	
\end{align}

Now, we proceed to estimate the sums $\Sigma_i(C,K),\, i=1,2$.  Bearing in mind \eqref{biquadratic-sol} one gets 
\begin{equation}\label{proof-crucial-lem-02}
\begin{split}
\Sigma_1(C, K)&=\sum_{|n|=0}^{\lfloor x^*(C) \rfloor}\frac{K} {\sqrt{C+K - \omega(n)} + \sqrt{C- \omega(n)}}\\
&\le\sum_{|n|=0}^{\lfloor x^*(C) \rfloor -1}\frac{K} {\sqrt{C- \omega(n)}} + \frac{2K}{\sqrt{C+K - \omega(\lfloor x^*(C) \rfloor)}}\\
&\le 2\sum_{n=0}^{\lfloor x^*(C) \rfloor -1}\frac{K} {\sqrt{C- \omega(n)}} + 2\sqrt{K}\\
&\le 2K\underbrace{\int_0^{x^*(C)}\!\!\!\!\frac{dz}{\sqrt{C- \omega(z)}}}_{\cJ(C)} + 2\sqrt{K}
\end{split}
\end{equation}

Further, making the change of variable $z= \frac{1}{\sqrt{-2\alpha}} \big(\sqrt{ 1 - 4\alpha C\rho} - 1)^{1/2}$ and using again \eqref{biquadratic-sol} we estimate $\cJ(C)$ as follows:
\begin{equation}\label{proof-crucial-lem-03}
\begin{split}
\cJ(C)&=\sqrt{\frac{-\alpha C}{2}}\int_0^1\frac{d\rho}{\sqrt{1-\rho}\,\big(\sqrt{1-4\alpha C\rho} -1\big)^{1/2}\sqrt{1-4\alpha C\rho}}\\
&=\sqrt{\frac{-\alpha C}{2}}\int_0^1\frac{\big(\sqrt{1-4\alpha C\rho} +1\big)^{1/2}}{\sqrt{1-\rho}\,\sqrt{-4\alpha C\rho}\, \sqrt{1-4\alpha C\rho}}d\rho\\
&\lesssim \int_0^1\frac{d\rho}{\sqrt{1-\rho}\,\sqrt{\rho}\, \big(1-4\alpha C\rho\big)^{1/4}}\\
&\lesssim \int_0^1\frac{d\rho}{\sqrt{1-\rho}\,\sqrt{\rho}}=\pi.
\end{split}
\end{equation}

Hence, combining \eqref{proof-crucial-lem-03} and \eqref{proof-crucial-lem-02} we get 
\begin{equation}\label{proof-crucial-lem-04}
\Sigma_1(C, K) \lesssim K.
\end{equation}

On the other hand,
\begin{equation}\label{proof-crucial-lem-05}
\begin{split}
\Sigma_2(C,K) &= 2\sum_{n=\lfloor x^*(C)\rfloor + 1}^{\lfloor x^*(C+K)\rfloor}\sqrt{C+K - \omega(n)}\\
&\le 2\sqrt{C+K - \omega(\lfloor x^*(C)\rfloor + 1)} + 2\int_{\lfloor x^*(C)\rfloor + 1}^{\lfloor x^*(C+K)\rfloor}\sqrt{C+K - \omega(z)}dz\\
&\le 2\sqrt{K} + 2\int_{x^*(C)}^{x^*(C+K)}\sqrt{C+K - \omega(z)}dz\\
&\le 2\sqrt{K} + 2\sqrt{K}\big(x^*(C+K) -  x^*(C)\big),
\end{split}
\end{equation}
and 
\begin{equation}\label{proof-crucial-lem-06}
\begin{split}
x^*(C+K) -  x^*(C)&= \frac{1}{\sqrt{-2\alpha}} \Big( \big(\sqrt{1 - 4\alpha(C+K)} - 1)^{1/2} - \big(\sqrt{1 - 4\alpha C} - 1)^{1/2}\Big)\\
&= \frac{2}{\sqrt{-2\alpha}}\left(\frac{\sqrt{-\alpha(C+K)}}{\big(\sqrt{1 - 4\alpha(C+K)} + 1\big)^{1/2}} - \frac{\sqrt{-\alpha C}}{(\sqrt{1 - 4\alpha C} + 1)^{1/2}}\right)\\
&\le\sqrt{\frac{2}{-\alpha}}\frac{\sqrt{-\alpha(C+K)} - \sqrt{-\alpha C}}{\big(\sqrt{1 - 4\alpha(C+K)} + 1\big)^{1/2}}\\
&=\sqrt{\frac{2}{-\alpha}}\frac{-\alpha K}{\big(\sqrt{1 - 4\alpha(C+K)} + 1\big)^{1/2} \big(\sqrt{-\alpha(C+K)} + \sqrt{-\alpha C}\big)}\\
&\le\sqrt{2K}. 
\end{split}
\end{equation}

Finally, combining \eqref{proof-crucial-lem-06} in  \eqref{proof-crucial-lem-05} we have 
\begin{equation}\label{proof-crucial-lem-07}
	\Sigma_2(C, K) \lesssim K.
\end{equation}
Thus, using the estimates \eqref{proof-crucial-lem-04} and   \eqref{proof-crucial-lem-07} in \eqref{proof-crucial-lem-01} we complete the proof of Lemma \ref{crucial-lem}.
\end{proof}

\subsection{Proof of Theorem \ref{Th-IP-below-L2}}  For the sake of clarity we provide details of the proof for $\varepsilon =1$. The case $\varepsilon =0$ follows in a similar way. We begin by obtaining explicit standing-wave solutions of \eqref{model}  in the focusing regime ($-|u|^2u$). Indeed, if we look for solutions in the form 
\begin{equation}\label{standing-wave}
	u(t,x,y)= e^{i\theta t}e^{in x}\varphi(y),
\end{equation}
where $\theta \in \R$, $n\in \Z$  and $\varphi: \R \to \R$ is a smooth localized function we obtain the following nonlinear ODE
\begin{equation}\label{sw-profile-ode}
	-\varphi''(y) + (n^2-\alpha n^4  +\theta)\varphi(y)  -\varphi^3(y)=0,
\end{equation}
satisfied by $\varphi$.

Recall that, we have considered $\alpha<0$. With this consideration, one has
\begin{equation}\label{sw-profile-ode-cond}
	\sigma_{n,\theta}:=n^2-\alpha n^4+ \theta >0.
\end{equation}

 It is easy to verify that the function 
\begin{equation}\label{sw-profile-ode-sol}
	\varphi(y)= \sqrt{2\sigma_{n,\theta}}\sech\big(\sqrt{\sigma_{n,\theta}}\, y\big)
\end{equation}
is a solution of \eqref{sw-profile-ode}. Hence, 
\begin{equation}\label{sw-sol}
	u_{n,\theta}(t,x,y)=\sqrt{2\sigma_{n,\theta}} e^{i\theta t}e^{in x}\sech\big(\sqrt{\sigma_{n,\theta}}\, y\big)
\end{equation}
under the constraint \eqref{sw-profile-ode-cond}.

We will prove that the  family of solutions \eqref{sw-sol} is not uniformly continuous from  $H^s(\T\times\R)$ into  the space $C\big([0, T];\, H^s(\T\times\R)\big)$ whenever $-\frac12\leq s<0$. 
	
	Consider a sequence $\gamma_n>0$ such that $\gamma_n\to\gamma>0$. For $\gamma_n>0$ with this property,  define
	\begin{equation}\label{not-2}
		\begin{cases}
			\theta_{{\gamma_n},n}:=\gamma_n^2n^{-4s}-n^2+\alpha n^4,\medskip\\
			\theta_{{\gamma},n}:=\gamma^2n^{-4s}-n^2+\alpha n^4.
		\end{cases}
	\end{equation} 
	
Taking in consideration the sequences  $\theta_{{\gamma_n},n}$ and $\theta_{{\gamma},n}$ defined in \eqref{not-2}, in view of  \eqref{sw-sol}, one obtains the following special solutions to the IVP \eqref{model}
	\begin{equation}\label{spl-sol2}
		\begin{cases}
			u_{{\gamma_n},n} = \sqrt{2} e^{i\theta_{{\gamma_n},n}t}e^{inx}\gamma_nn^{-2s}\sech(\gamma_nn^{-2s}y),\medskip\\
			u_{{\gamma},n} = \sqrt{2} e^{i\theta_{{\gamma},n}t}e^{inx}\gamma\, n^{-2s}\sech(\gamma\, n^{-2s}y).
		\end{cases}
	\end{equation}
	Observe that, the Fourier transforms of the special solutions $u_{{\gamma_n},n}$ and $u_{{\gamma},n}$ constructed in \eqref{spl-sol2} are given by
	\begin{equation}\label{F-T1}
		\widehat{u}_{{\gamma_n},n}(k,\xi) = \begin{cases}0\qquad& \text{if}\; k\ne n,\medskip \\
			\sqrt{2}\;\widehat{\sech}\Big(\frac{\xi}{\gamma_n n^{-2s}}\Big) & \text{if}\; k= n,
		\end{cases}
	\end{equation}
	and
	\begin{equation}\label{F-T2}
		\widehat{u}_{{\gamma},n}(k,\xi) = \begin{cases}0\qquad& \text{if} \; k\ne n,\medskip\\
			\sqrt{2}\;\widehat{\sech}\Big(\frac{\xi}{\gamma n^{-2s}}\Big)& \text{if} \; k= n.
		\end{cases}
	\end{equation}

	In what follows, we calculate the $H^s(\T\times \R)$-norm  of  $u_{{\gamma_n},n}$. Using \eqref{F-T1}, one easily obtains
	\begin{equation}\label{Hs-norm1}
		\begin{split}
			\|u_{{\gamma_n},n} \|_{H^s}^2&=\sum_{k\in\Z}\int_{\R}(1+|\xi|+|k|)^{2s}|\widehat{u_{{\gamma_n},n}}(k,\xi) |^2d\xi\\
			&=2\int_{\R}(1+|\xi|+n)^{2s}\Big|\widehat{\sech}\Big(\frac{\xi}{\gamma_n n^{-2s}}\Big)\Big|^2d\xi\\
			&=2\alpha_n\int_{\R}\Big(\frac1n+\frac{\gamma_n}{n^{1+2s}}|\eta|+1\Big)^{2s}|\widehat{\sech}(\eta)|^2d\eta.
		\end{split}
	\end{equation}
	
	Now, from \eqref{Hs-norm1} one can easily infer that 
	\begin{equation}
		\label{Hs-norm2}
		\|u_{{\gamma_n},n} \|_{H^s}\leq C_{\gamma}\|\sech(\cdot)\|_{L^2},
	\end{equation}
for all $n\in\N$ and $s<0$,

	Also, using \eqref{F-T1} and \eqref{F-T2}, we obtain
	\begin{equation}
		\label{Hs-conv1}
		\begin{split}
			\|u_{{\gamma_n},n}(0)-u_{{\gamma},n}(0) \|_{H^s}^2&= 2\int_{\R}(1+|\xi|+n)^{2s}\Big|\widehat{\sech}\Big(\frac{\xi}{\gamma_n n^{-2s}}\Big)- \widehat{\sech}\Big(\frac{\xi}{\gamma \,n^{-2s}}\Big)\Big|^2d\xi\\
			&=2\gamma \int_{\R}\Big(\frac1n+\frac{\gamma_n}{n^{1+2s}}|\eta|+1\Big)^{2s}|\widehat{\sech}\big(\frac{\gamma}{\gamma_n}\eta\big) - \widehat{\sech}(\eta)|^2d\eta\\
			&\leq 2\gamma\int_{\R}|\widehat{\sech}\big(\frac{\gamma}{\gamma_n}\eta\big) - \widehat{\sech}(\eta)|^2d\eta,
		\end{split}
	\end{equation}
	where $s<0$ has been used in the last step.  
	
	Hence, using the Dominated Convergence Theorem,  for $s<0$, we can conclude from \eqref{Hs-conv1} that
	\begin{equation}
		\label{Hs-conv2}
		\lim\limits_{n\to +\infty}\|u_{{\gamma_n},n}(0)-u_{{\gamma},n}(0)\|_{H^s(\R\times\T)}=0.
	\end{equation}
	
	Now, we move on to estimate from below  the $H^s$-norm of the difference of the respective evolutions $u_{{\gamma_n},n}(t)$  and $u_{{\gamma},n}(t)$ of $u_{{\gamma_n},n}(0)$ and $u_{{\gamma},n}(0)$.
	
	First, note that
	\begin{equation}
		\begin{split}
			\label{Hs-apart}
			\|u_{{\gamma_n},n}(t)-u_{{\gamma},n}(t) \|_{H^s}&\geq \|u_{{\gamma_n},n}(t)-W_{\gamma,{\gamma_n}}(t)\|_{H^s}-\|W_{\gamma,{\gamma_n}}(t)-u_{{\gamma},n}(t) \|_{H^s}\\
			&=:A(n,t)-B(n,t),
		\end{split}
	\end{equation}
	where 
	\begin{equation}\label{Un}
		W_{\gamma,{\gamma_n}}(x,y, t):= \sqrt{2}e^{i\theta_{\gamma,n}t}e^{inx}2\gamma_nn^{-2s}\sech(2\gamma_nn^{-2s}y).
	\end{equation}

	In sequel, we estimate the terms $A(n,t)$ and $B(n,t)$ appearing in the RHS of \eqref{Hs-apart}.
	
	 Note that, the estimate for the second term $B(n,t)$ satisfies the following property
	\begin{equation}
		\label{Hs-2nd}
		\begin{split}
			B(n,t)&=\|W_{\gamma,{\gamma_n}}(t)-u_{{\gamma},n}(t) \|_{H^s}\\
			&=\|e^{i\theta_{\gamma,n}t}\Big(u_{{\gamma_n},n}(0)-u_{{\gamma},n}(0)\Big)\|_{H^s}
			=\|u_{{\gamma_n},n}(0)-u_{{\gamma},n}(0)\|_{H^s}\leq \nu_n,
		\end{split}
	\end{equation}
	where $\nu_n\to 0$ as $n\to +\infty$ by \eqref{Hs-conv2}. 
	
On the other hand,  the first term $A(n,t)$, for  $t>0$ fixed, enjoys the following lower bound
	\begin{equation}\label{Hs-LB}
		\begin{split}
			A(n,t)&=\|u_{{\gamma_n},n}(t)-W_{\gamma,{\gamma_n}}(t)\|_{H^s}\\
			&=\Big(\sqrt{2}\int_{\R}(1+|\xi|+n)^{2s}\Big|e^{i\theta_{\gamma_n, n}t}-e^{i\theta_{\gamma, n}t}\Big|^2\Big|\widehat{\sech}\Big(\frac{\xi}{\gamma_nn^{-2s}}\Big)\Big|^2\Big)^{\frac12}\\
			&=2^{\frac14}\sqrt{\gamma_n}\Big|e^{i\theta_{\gamma_n, n}t}-e^{i\theta_{\gamma, n}t}\Big|\Big(\int_{\R}\Big(\frac{n}{1+\gamma_nn^{-2s}|\eta|+n}\Big)^{-2s}|\sech(\eta)|^2d\eta\Big)^{\frac12}.
		\end{split}
	\end{equation}

	In order to obtain a lower bound for the integral in \eqref{Hs-LB},  we define, for $z\geq 1$
	\begin{equation}\label{F-z}
		F(z):=\int_{\R}g(z, \eta)|\;\widehat{\sech}(\eta)|^2d\eta,
	\end{equation}
	where
	\begin{equation}
		\label{g-z}
		g(z,\eta):= \Big(\frac{z}{1+\gamma_n|\eta|z^{-2s}+z}\Big)^{-2s}.
	\end{equation}
	
	It is easy to check that
	\begin{equation}\label{d-gz}
		\partial_zg(z,\eta)= -2s\Big(\frac{z}{1+\gamma_n|\eta|z^{-2s}+z}\Big)^{-2s-1}\; \frac{(1+2s)\gamma_n|\eta|z^{-2s}}{\big({1+\gamma_n|\eta|z^{-2s}+z}\big)^{2}}\geq 0
	\end{equation}
for all $-\frac12\leq s<0$, independent of $\eta$. Hence, it can be deduced that 
	\begin{equation}\label{g-inc}
		g(n,\eta)\geq g(1, \eta), \qquad \forall \; \eta\in \R,
	\end{equation}
	and consequently $F(n)\geq F(1)$. 
	
	Now, using  this information in \eqref{Hs-LB}, one gets
	\begin{equation}\label{Hs-Ist}
		A(n,t)\geq C_{\gamma}\big|e^{i\theta_{\gamma_n, n}t}-e^{i\theta_{\gamma, n}t}\big|=C_{\gamma}\big|e^{it(\theta_{\gamma_n, n}-\theta_{\gamma, n})}-1\big|.
	\end{equation}

In view of estimates \eqref{Hs-2nd} and \eqref{Hs-Ist}, one obtains from \eqref{Hs-apart} that
	\begin{equation}
		\begin{split}
			\label{Hs-apart2}
			\|u_{{\gamma_n},n}(t)-u_{{\gamma},n}(t) \|_{H^s}&\geq C_{\gamma}\big|e^{it(\theta_{\gamma_n, n}-\theta_{\gamma, n})}-1\big|-\nu_n\\
			&=C_{\gamma}\big|e^{-it(\gamma^2-\gamma_n^2)n^{-4s}}-1\big|-\nu_n.
		\end{split}
	\end{equation}
	
Now, if one chooses the sequence $\gamma_n$ satisfying $(\gamma^2-\gamma_n^2)n^{-4s}=\tau n^{\delta}$, for some $\tau>0$ and $\delta>0$ such that $4s+\delta<0$, it can easily be seen that $\gamma_n\to\gamma$ when $n\to +\infty$. On the other hand, for this choice, for some fixed $t>0$, one has
	\begin{equation}\label{donot-conv}
		\big|e^{-it(\gamma^2-\gamma_n^2)n^{-4s}}-1\big|\nrightarrow 0, \quad {\mathrm as}\;\; n\to+\infty.
	\end{equation}

Recall from \eqref{Hs-2nd} that $\nu_n\to 0$ as $n\to+\infty$,  and in the choice of the sequence $\gamma_n$ the parameter $\delta >0$ is arbitrary. Hence, from \eqref{Hs-conv2},  \eqref{Hs-apart2} and \eqref{donot-conv} one can conclude  that the evolution of the initial data that are very close in $H^s$-norm do not stay close for the time $t>0$, whenever $-\frac12\leq s<0$. Therefore, the mapping data-solution is not uniformly continuous on the bounded sets of initial data in $H^s(\T\times \R)$ whenever $-\frac12\leq s<0$, thereby completing the proof of the theorem. 

\setcounter{equation}{0}
\section{Proof of results for cylindrical domain  $\R\times \T$}\label{Section-RxT}

This section is devoted in addressing the well-posedness issues for the IVP \eqref{model} posed on the cylindrical domain  $\R\times \T$. In this case too, we provide details considering $\varepsilon =1$. The case $\varepsilon =0$ follows with simple modifications.

We start reviewing some results obtained by S. Cui and C. Guo \cite{CG2007}, for the  IVP  associated with the following  fourth-order NLS equation
\begin{equation}\label{Eq1}
\begin{cases}
    i \partial_{t} u + a \Delta u + b \Delta^{2} u = c |u|^{\sigma} u,& (t,x) \in [0,T] \times  \mathbb{R}^{n} \medskip \\
    u(0,x) = \phi(x),
\end{cases}
\end{equation}
where $a,b, c$ are real constants, $b \neq 0, c \neq 0$, and $\sigma >0$. The authors in \cite{CG2007} obtained the well-posedness results considering the $n$-dimensional case with general nonlinearity. Strichartz estimates for the associated elliptic operator were the main ingredient in the proof.

While dealing with the IVP \eqref{model} posed on $\cD=\R\times \T$, we will consider the one dimensional version of the Strichartz estimates in the $x$-variable proved in \cite{CG2007} and use it to deal with the periodic $y$-variable as a perturbation in each Fourier mode. In what follows, we review  Strichartz estimates from \cite{CG2007} associated to the IVP \eqref{Eq1} for  $n = 1$ and $\sigma = 2$.

\subsection{Review of the Strichartz estimates in purely continuous case}
We are interested in the case  $n = 1$ in this section. Let us denote by $S_{a,b}(t) \phi$ the free-solution for \eqref{Eq1}. The  associated linear operator is $ L_{a, b, x}:= (a \partial_{x}^{2} + b \partial_{x}^{4})$,  whose Fourier symbol is $- a \xi^{2} + b \xi^{4}$. We say that $(q,p)$ is an admissible pair if it satisfies
\begin{equation}\label{Eq2}
 \frac{1}{q} = \frac{1}{4} \Big(\frac{1}{2} - \frac{1}{p}\Big) \Leftrightarrow   \frac{1}{p} + \frac{4}{q} = \frac{1}{2},
\end{equation}
 with the assumption that $2 \leq p \leq \infty$. Note that, in \eqref{Eq2}, $ p = \infty$ implies $q = 8$ and $p = 2$ implies $q = \infty$. If ($p \neq 2, \infty$) we  have $q = \frac{8p}{p-2}$. Thus, $8 \leq q \leq \infty$. Moreover, if  $(\gamma, \rho)$ is an admissible pair then its conjugate pair $(\gamma', \rho')$ satisfies the following conditions. Since
\[
\frac{1}{\rho} + \frac{1}{\rho'} = 1 \Leftrightarrow \rho' = \frac{\rho}{\rho - 1}, 
\]
our consideration $2 \leq \rho \leq \infty$  implies that $1 \leq \rho' \leq 2$. On the other hand, 
\[
\frac{1}{\gamma} + \frac{1}{\gamma'} = 1 \Leftrightarrow \gamma' = \frac{\gamma}{\gamma - 1},
\]
with $8 \leq \gamma \leq \infty$,   so that $1 \leq \gamma' \leq \frac{8}{7}$. 

\medskip 
Using the notation established above, the following results hold:

\begin{proposition}\label{Strichartz-1d} Let $T_0>0$ and $0<T\le T_0$. If  $(q,p)$ and $(\gamma, \rho)$ are admissible pairs, 
then we have the following Strichartz estimates:
\begin{equation}\label{Eq3}
   \|S_{a,b}(t) \phi \|_{L_{T}^{q} L_{x}^{p}} \leq C \|\phi\|_{L^{2}} ,
\end{equation}

\begin{equation}\label{Eq4}
  \sup_{|t| \leq T} \Big\|\int_{0}^{t} S_{a,b}(t - \tau) f(\cdot, \tau) d \tau \Big\|_{2} \leq C \|f \|_{L_{T}^{q'} L_{x}^{p'}} , 
\end{equation}

\begin{equation}\label{Eq5}
\Big\| \int_{0}^{t} S_{a,b}(t - \tau) f(\cdot, \tau) d \tau \Big\|_{L^{q}_{T} L_{x}^{p}} \leq C \|f \|_{L_{T}^{\gamma'} L_{x}^{\rho'}},
\end{equation}
where the constant $C$ depends on $p$, $q$ and $T_0$. 
\end{proposition}

\begin{proof}
For detailed proof we refer to  \cite{CG2007}, more precisely, Theorem 2.2, Theorem 2.4 and Theorem 2.5 respectively in pages  690, 692 and 694 there.
\end{proof}

We also record the following result, whose proof can be found in \cite[Corollary 2.6, p.~694]{CG2007}.
\begin{proposition}\label{DerivativeStrichartz}  For any real $s$ and any admissible pairs $(q,p)$, $(\gamma, \rho)$ we have  the following estimates 
\begin{equation}\label{CorollaryStrichartzI}
 \|S_{a,b}(t) h\|_{L^{q}(\mathbb{R}, W^{s,p})} \leq C \|h \|_{H^{s}}   
\end{equation}
and 
\begin{equation}\label{CorollaryStrichartzII}
\left\| \int_{0}^{t} S_{a,b}(t - \tau) h(\cdot, \tau) d \tau \right\|_{L^{q}((-T, T), W^{s,p})} \leq C \|h \|_{L^{\gamma'}((-T, T), W^{s, \rho'})} . 
\end{equation}
\end{proposition}

\medskip
\subsection{Auxiliary results for well-posedness in $H^{s}(\mathbb{R} \times \mathbb{T})$}

We define the operator associated to the linear part of the IVP \eqref{model}
\begin{equation}\label{Eq6}
 Q_{\alpha, x,y} := \partial_{x}^{2} + \alpha \partial_{x}^{4} + \partial_{y}^2 = L_{\alpha,x}  +  \partial_{y}^2, 
\end{equation}
for $\alpha \neq 0$ and $(x,y) \in \mathbb{R} \times \mathbb{T}$. Recall that, we are taking $\varepsilon =1$. Considering  $(a, b) =(1, \alpha)$ in the Strichartz estimates \eqref{Eq3}-\eqref{Eq5} on $\mathbb{R}$, in this section, we deduce new Strichartz-type estimate for the associated free solution  $e^{i t Q_{\alpha, x,y}}\phi$. Idea of the proof is based on the work of Tzvetkov-Visciglia \cite{TV-12} where the authors considered NLS equation on product spaces. In our case, the asymmetric structure of the symbol associated to  the linear problem forced to adapt several indices. On the other hand, we are working on a simple solution space compared to the one used in \cite{TV-12}.

\begin{proposition}\label{PropI}
The following estimate holds:
\begin{multline}\label{Eq7-a}
   \|D^{k}e^{i t Q_{\alpha, x,y}} f\|_{L^{q}_{t} L_{x}^{p} H_{y}^{\ell}} + \Big\| D^{k} \int_{0}^{t} e^{i (t - \tau)  Q_{\alpha, x,y}} F(\tau, x, y) d \tau \Big\|_{L^{q}_{t} L_{x}^{p} H_{y}^{\ell}}\\  \leq  C (\|D^{k}f \|_{L^{2}_{x} H_{y}^{\ell}} +\|D^{k}F \|_{L^{\gamma'}_{t} L_{x}^{\rho'} H_{y}^{\ell}} ), 
\end{multline}
where $D^k = \partial_{x}^k , \partial_{y}^k$ ($k = 0,1$), $C = C(p, q, \gamma, \rho) >0$,  $\ell \geq 0$,  and $(q, p)$, $(\gamma, \rho)$ are admissible pairs, i.e.,
\[
\frac{1}{p} + \frac{4}{q} = \frac{1}{2} =  \frac{1}{\gamma} + \frac{4}{\rho}, 
\]
with $2 \leq p , \rho \leq \infty$.   
\end{proposition}

\begin{proof} 
First, consider the case $k=0$ and $\ell= 0$.  Note that, Proposition \ref{Strichartz-1d}  gives the following Strichartz estimate for the free propagators $e^{i t (L_{\alpha,x} + \beta)}$ on $\mathbb{R}$ with $\beta \in \mathbb{R}$:

\begin{equation}\label{Eq7}
\sup_{\beta  \in \mathbb{R}} \left(\|e^{i t (L_{\alpha,x} + \beta )} h \|_{L_{T}^{q} L_{x}^{p}}  + \left\| \int_{0}^{t} e^{i (t - \tau) (L_{\alpha,x} + \beta )}  H(\tau, x) d \tau \right\|_{L_{T}^{q} L_{x}^{p}}   \right ) \leq C (\| h \|_{L^{2}_{x}} + \|H\|_{L_{T}^{\gamma'} L_{x}^{\rho'}}),
\end{equation}
where $L_{\alpha,x} := \partial_{x}^2 + \alpha \partial_{x}^{4}$  and  $C = C(p, \rho, q, \gamma) >0$.

Further, we introduce 
\begin{equation}\label{DuhamelFormula}
u(t, x,y) = e^{i t Q_{\alpha,x,y}} f + \int_{0}^{t} e^{i (t - \tau) Q_{\alpha,x,y}} F(\tau, x, y) d \tau,    
\end{equation}
and notice that 
\begin{equation}\label{eq-Dh}
\begin{cases}
i \partial_{t} u +  Q_{\alpha,x,y} u = F, \qquad (t,x,y) \in \mathbb{R} \times \mathbb{R} \times \mathbb{T}, \medskip\\
u(0, x,y) = f(x,y). 
\end{cases}
\end{equation}

Now, we expand $u, f$ and $F$ using Fourier series with respect to the orthonormal basis $\{ e^{ i y n} \}_{ n \in \mathbb{Z}}$ of  $ L^{2}(\mathbb{T})$, to obtain

\begin{equation}\label{Eq8}
  u(t, x, y) = \sum_{n \in \mathbb{Z}} \widehat{u}(t, x, n) e^{ i y n},
\end{equation}

\begin{equation}\label{Eq9}
     F(t, x, y) = \sum_{n \in \mathbb{Z}} \widehat{F}(t, x, n) e^{ i y n}, 
\end{equation}
and 
\begin{equation} \label{Eq9-x}
f(x,y) =\sum_{n\in \mathbb{Z}}  \widehat{f}( x,n) e^{ i y n}.
\end{equation}

Notice that $\widehat{u}(t, x, n)$, $\widehat{F}(t, x, n)$ and $\widehat{f}(t,x,n)$ are related by the following family of IVPs:

\begin{equation}\label{Eq10}
\begin{cases}
i \partial_{t} \widehat{u}(t, x, n)  + (L_{\alpha, x} - n^{2}) \widehat{u}(t, x,n)  =\widehat{F}(t, x,n) , \; (t,x,n) \in \mathbb{R}\times \mathbb{R} \times \mathbb{Z}, \medskip \\  
 \widehat{u}(0, x, n) = \widehat{f}(x, n). 
\end{cases}
\end{equation}

In order to  simplify  the notations, let us write $u_{n} :=\widehat{u}(t, x, n)$, $F_{n} :=\widehat{F}(t, x, n) $ and  $f_{n} :=\widehat{f}(x, n) $. Applying \eqref{Eq7}  for \eqref{Eq10}, we obtain
\begin{equation}\label{Eq11}
\|u_{n}(t,x)\|_{L_{T}^{q} L_{x}^{p}} \lesssim  (\| f_{n} \|_{L^{2}_{x}} + \|F_{n}(t,x)\|_{L_{T}^{\gamma'} L_{x}^{\rho'}}),
\end{equation}
where $(q,p)$ and $(\rho, \gamma)$ are admissible pairs. 

Now, summing in $n$ the squares,  \eqref{Eq11} gives us
\begin{equation}\label{eq11-x}
\|u_{n}(t,x)\|_{\ell^{2}_{n}L_{T}^{q} L_{x}^{p}} \leq  C (\| f\|_{L^{2}_{x,y}} + \|F_{n}(t,x)\|_{\ell_{n}^{2} L_{T}^{\gamma'} L_{x}^{\rho'}}). 
\end{equation}
On the other hand, since
\begin{equation*}
\max \{ \rho', \gamma'\} \leq 2 \leq \min \{ p, q \}, 
\end{equation*} 
using the Minkowski inequality one gets from \eqref{eq11-x} that
\begin{equation}\label{eq11-xx}
\|u_{n}(t,x)\|_{L_{T}^{q} L_{x}^{p}\ell^{2}_{n}} \leq  C (\| f\|_{L^{2}_{x,y}} + \|F_{n}(t,x)\|_{ L_{T}^{\gamma'} L_{x}^{\rho'}\ell^{2}_{n}}). 
\end{equation}

Finally, using \eqref{Eq9}, \eqref{Eq10} and Plancherel's identity, we have 
\begin{equation}\label{eq11-y}
\|u\|_{L_{T}^{q} L_{x}^{p}L^{2}_{y}} \leq  C (\| f\|_{L^{2}_{x,y}} + \|F\|_{ L_{T}^{\gamma'} L_{x}^{\rho'}L^{2}_{y}}). 
\end{equation} 

To obtain the estimate \eqref{Eq7-a} with $k=\ell=0$ we apply the last inequality \eqref{eq11-y} first with $F=0$ to get the linear estimate and then with  $f=0$ to derive the estimate for the non-homogeneous term.

 The general case can be addressed with some modifications as follows. In the periodic $y$-variable, everything can be reduced to the case $\ell =0$ because we can introduce the operator $J_{y}^{\ell}:= (1 - \partial_{y}^{2})^{\frac{\ell}{2}}$ into  both sides of equation  \eqref{DuhamelFormula} which commutes with the linear group $e^{it Q_{\alpha, x, y}}$. 
 
   The case for $k \neq 0$ and $\ell =0$ can be treated splitting in two different parts. In the case when $D^k = \partial_x^k$ we perform the above procedure  using Proposition \ref{DerivativeStrichartz} instead of Proposition \ref{Strichartz-1d}. Indeed, in view of \eqref{CorollaryStrichartzI} and \eqref{CorollaryStrichartzII} we have the estimate
\begin{equation*}
\sup_{\beta \in \mathbb{R}} \left(\|\partial_{x}^{k}e^{i t (L_{\alpha,x} + \beta)} h \|_{L_{T}^{q} L_{x}^{p}}  + \left\| \partial_{x}^{k} \int_{0}^{t} e^{i (t - \tau) (L_{\alpha,x} + \beta)}  H(\tau, x) d \tau \right\|_{L_{T}^{q} L_{x}^{p}}   \right ) \leq C (\| h \|_{H^{k}_{x}} + \|\partial_{x}^{k}H\|_{L_{T}^{\gamma'} L_{x}^{\rho'}}),
\end{equation*}
where we assumed the same conditions on $p, \rho, q, \gamma$ and $C = C(p, \rho, q, \gamma) >0$.

 On the other hand, in the case when $D^{k} = \partial_{y}^{k}$ we proceed by introducing $\partial_{y}^{k}$ into \eqref{DuhamelFormula} and make similar considerations as those used for $J_{y}^{\ell}$. This completes the proof of Proposition \ref{PropI}.
\end{proof}
  
 \medskip
 The  following consequence of the Proposition \ref{PropI} will be useful. 

\begin{corollary}\label{Prop3}
Let $(q,p)$ and $(\gamma, \rho)$ admissible pairs. For any $s \geq 0$  we have 
\[
\|J_{y}^{s} e^{i t Q_{\alpha, x,y} }f \|_{L^{q}_{t} L_{x}^{p} L_{y}^{2}} + \| J_{y}^{s} \Big(\int_{0}^{t} e^{i (t - \tau) Q_{\alpha,x,y}} F(\tau) d \tau \Big)\|_{L_{t}^{q'} L_{x}^{p'} L_{y}^{2}}
\] 
\[
\lesssim \|f\|_{\mathcal{H}_{x,y}^{0,s}} + \| J^{s}_{y} F\|_{L_{t}^{\gamma'} L_{x}^{\rho'} L_{y}^{2}},
\]
where $C = C(q,p,\gamma, \rho) >0$.
\end{corollary}

 Now, we are in position to prove the local well-posedness result for given data in the anisotropic Sobolev space on domain $\R\times \T$.
Our goal is to establish local well-posedness in the space 
\begin{equation}\label{NI-FixedPoint-Space}
Z_{T}^{0, s} :=L_{T}^{\infty}\mathcal{H}_{x,y}^{0,s } \cap G_{x,y}^{0, s}
\end{equation}
with $s>1/2$, where $G_{x,y}^{0, s}$ is the  auxiliary space defined by the norm
\[
\|f\|_{G_{x,y}^{0, s}} = \|J_{y}^{s} f\|_{L_{T}^{12} L_{x}^{6} L_{y}^{2}}.
\]
Notice that, $(12, 6)$ is an admissible pair, i.e., verifies \eqref{Eq2}.    

\subsection{Proof of Theorem \ref{Th-LWP-aniso}}
 Let $(12, 6)$ be an admissible pair and let  $\phi \in \mathcal{H}_{x,y}^{0,s}$ with  $s>\frac12$. Consider  the space $Z_{T}^{0, s}$ defined in 
\eqref{NI-FixedPoint-Space} endowed with the norm

\begin{equation}\label{Zs-norm}
\|u\|_{Z_{T}^{0, s}} = \sup_{|t|\leq T}\|u\|_{\mathcal{H}_{x,y}^{0, s }}+ \|u\|_{G_{x,y}^{0, s}}.
\end{equation}

\medskip 
Consider now the integral formulation associated to the IVP  \eqref{model}:
\[
\Phi_{\phi} (u):= e^{i t Q_{\alpha,x,y}} \phi + \int_{0}^{t} e^{i (t - \tau) Q_{\alpha,x,y}} (u(\tau) |u(\tau)|^{2}) d \tau. 
\]
We will show that for all  $\phi \in \mathcal{H}_{x,y}^{0,s}$, with $s>\frac12$, there exist a positive time $T = T(\|\phi\|_{\mathcal{H}_{x,y}^{0,s }})$ and 
$R = R(\|\phi\|_{\mathcal{H}_{x,y}^{0,s}}) >0$ such that 
$$\Phi_{\phi} (B_{Z_{T'}^{0, s}}) \subset B_{Z_{T'}^{0, s}},\quad\; T' < T.$$

From Corollary \ref{Prop3}, we obtain
\begin{equation}\label{eq-cont1s}
\begin{split}
\|\Phi_{\phi}(u)\|_{ \mathcal{H}_{x,y}^{0,s}} &\leq C_0 \|\phi\|_{ \mathcal{H}_{x,y}^{0,s}} +\int_0^T \| u |u|^2\|_{ \mathcal{H}_{x,y}^{0,s}} d\tau.
\end{split}
\end{equation}

Recalling that $H_y^{s_2}$ for $s_2>\frac12$ is a Banach algebra and an use of  H\"older's inequality yield
\begin{equation}\label{eq-cont2s}
\begin{split}
\int_0^T \|u|u|^2 \|_{ \mathcal{H}_{x,y}^{0,s}}d\tau = \int_0^T \|\,\| u\|_{H_y^{s}}^3\|_{L_x^2} d\tau
=\int_0^T \|\,\| u\|_{H_y^{s}}\|_{L_x^6}^3d\tau
\leq C T^{\frac34} \| u\|_{L_T^{12}L_x^6H_y^{s}}^3.
\end{split}
\end{equation}

Now, inserting  \eqref{eq-cont2s} in \eqref{eq-cont1s} and recalling the definition of $ {G}_{x,y}^{0,s}$-norm, one has
\begin{equation}\label{eq-cont2.1s}
\begin{split}
\|\Phi_{\phi}(u)\|_{ \mathcal{H}_{x,y}^{0,s}}
&\leq C_0 \|\phi\|_{ \mathcal{H}_{x,y}^{0,s}}  +C T^{\frac34} \| u\|_{ G_{x,y}^{0,s}}^3.
\end{split}
\end{equation}

Once again, by  Corollary \ref{Prop3}, we have 
\begin{equation}\label{eq-cont3s}
\begin{split}
\|\Phi_{\phi}(u)\|_{G_{x,y}^{0,s}}& \leq C_0 \|\phi\|_{ \mathcal{H}_{x,y}^{0,s}} +\int_{0}^{T}\|e^{i (t - \tau) Q_{\alpha,x,y}} (u(\tau) |u(\tau)|^{2})\|_{L_T^{12}L_x^{6}H_y^{s}} d \tau\\
& \leq C_0 \|\phi\|_{\mathcal{H}_{x,y}^{0,s}} +C\int_{0}^{T}\|e^{-i \tau Q_{\alpha,x,y}} (u(\tau) |u(\tau)|^{2})\|_{ \mathcal{H}_{x,y}^{0,s}} d \tau\\
&= C_0 \|\phi\|_{\mathcal{H}_{x,y}^{0,s}} +C\int_{0}^{T}\| (u(\tau) |u(\tau)|^{2})\|_{ \mathcal{H}_{x,y}^{0,s}} d \tau.
\end{split}
\end{equation}

For the second term in the last line of \eqref{eq-cont3s} we can use the estimate from \eqref{eq-cont2s} to get
\begin{equation}\label{eq-cont5s}
\|\Phi_{\phi}(u)\|_{G_{x,y}^{0,s}} \leq C_0 \|\phi\|_{ \mathcal{H}_{x,y}^{0,s}} +C T^{\frac34} \| u\|_{G_{x,y}^{0,s}}^3.
\end{equation}

Combining \eqref{eq-cont2.1s} and \eqref{eq-cont5s}, one arrives at
\begin{equation}\label{eq-cont6m}
\|\Phi_{\phi}(u)\|_{Z_{T}^{0, s}} \leq C_0 \|\phi\|_{ \mathcal{H}_{x,y}^{0,s}} +C T^{\frac34} \| u\|_{Z_{T}^{0, s}}^3.
\end{equation}

Finally, choosing $R= 2C_0 \|\phi\|_{ \mathcal{H}_{x,y}^{0,s}}$ and $T= (2CR^2)^{-\frac43}$ we prove that, for small enough $T$, the application $\Phi_{\phi}$ maps a closed ball $B_{R}$ of radius $R >0$ of $Z_{T}^{0,s}$ into  itself. Moreover, it can be shown to be a contraction on the same ball by observing that
\[
\| \Phi_{\phi} (v) - \Phi_{\phi} (w) \|_{Z_{T}^{0,s}} \lesssim \|v - w \|_{Z_{T}^{0,s}} \Big(\|v \|_{Z_{T}^{0,s}} + \|w \|_{Z_{T}^{0,s}}\Big)^{2} \lesssim 
\|v - w \|_{Z_T^{0,s}} R^{2}. 
\]

The rest of the proof follows a standard argument, so we omit the details.

\setcounter{equation}{0}
\section{Proof of results for periodic domain  $\T^2$}\label{Section-TxT}

This section is devoted to study the IVP \eqref{model} posed on the purely periodic domain $\T^2$ and prove Theorem \ref{Th-LWP-T}. As pointed out in Section \ref{sec-2}, using standard arguments, the proof  of Theorem \ref{Th-LWP-T} is a consequence of Proposition \ref{Prop-St2}. Before presenting a proof of Proposition~\ref{Prop-St2} we record the following result on compact surface.

\begin{lemma}\label{lemma-CS}
Let $\alpha<0$ and consider the following compact surface 
\begin{equation}\label{h-surface}
\mathbb{S}_{\alpha} := \{ (\eta_1, \eta_2 ,  -\alpha \eta_1^4+ \eta_2^2) \in \mathbb{R}^3 : |\eta_i|\leq 1 \}.
\end{equation}
Then, $\mathbb{S}_{\alpha}$ possesses positive definite second fundamental form.
\end{lemma}

\begin{proof}

Let $\varphi(v,w)$ be a  canonical   parametrization of  $\mathbb{S}_{\alpha}$  given by
$$
\varphi(v, w) =
(v, w, -\alpha v^4+w^2), \qquad v^2+w^2\leq 1.
$$
For this parametrization, we have
\begin{equation}
\begin{split}
\varphi_{v}& = (1, 0,  -4 \alpha  v^{3})\\
\varphi_{w}& = (0, 1, 2w  )
\end{split}
\end{equation}

so that the coefficients of the first fundamental form are given by
\begin{equation}
\begin{split}
E &:= \langle \varphi_{v}, \varphi_{v}\rangle =1 + 16 \alpha^2  v^{6} \\
F &:= \langle \varphi_{w}, \varphi_{v}\rangle =  -8\alpha  v^3w \\
G&:= \langle \varphi_{w}, \varphi_{w}\rangle   = 1 +  4 w^2.
\end{split}
\end{equation}

Also
\[
\varphi_{v} \times \varphi_{w}  = (4\alpha v^3, - 2w , 1), 
\]
and therefore $\varphi$  defines the following field of unit normal vectors 

\[
{\bf N}(v,w) = \frac{1}{\sqrt{16\alpha^2 v^6 + 4 w^2 + 1 }} ( 4 \alpha u^{3}, - 2w, 1).
\]

To find the coefficients of the second fundamental form, we compute the second derivatives of the parametrization
\begin{equation}
\begin{split}
\varphi_{vv}& = (0, 0,  -12\alpha   v^2)\\
\varphi_{vw}& = (0, 0, 0)\\
\varphi_{ww} &= (0, 0, 2 ).\\
\end{split}
\end{equation}

Thus, the coefficients of the second fundamental form are given by
\begin{equation}
\begin{split}
e(v,w)&:=\langle \varphi_{vv}, {\bf N}\rangle = \frac{ -12\alpha   v^2}{\sqrt{16\alpha^2 v^6  + 4 w^2 + 1  }}\\
f(v,w)&:=\langle \varphi_{vw},{\bf N}\rangle = 0,\\
g(v,w)&:=\langle \varphi_{ww}, {\bf N}\rangle  =   \frac{2 }{\sqrt{16\alpha^2 v^6  + 4 w^2 + 1  }}.
\end{split}
\end{equation}

So the second fundamental form is given by 
\[
\Lambda(a\varphi_v+b\varphi_w)= a^2 e(v,w) + 2 ab f(v,w) + b^2 g(v,w) =  a^2 e(v,w) +  b^2 g(v,w), \quad (a,b)\in \R^2, 
\]
with $\{\varphi_v,  \varphi_w\}$  being a basis for $T_{(v,w)}(\mathbb{S}_{\alpha})$.  Since $f=0$, the matrix associated to the second fundamental is 
\begin{center}
$  \left(
  \begin{array}{cc}
  e(v,w)  &  0  \\
    0 &  g(v,w) \\
  \end{array}
\right) $
\end{center}
which is positive definite from Sylvester's criterion. Indeed for $\alpha<0$, we have that
$$e(v,w)>0$$ and 
\[
 eg(v,w)  =  \frac{-12 \alpha v^2 }{ 16\alpha^2 v^6 + 4 w^2 + 1  } >0
\]
for any $(v,w)$.

Finally, the Gaussian curvature given by 
\[
\frac{eg - f^2}{EG - F^2} =  \frac{-12 \alpha v^2 }{[ 16\alpha^2 v^6 + 4 w^2 + 1 ]^2 } >0
\]
is also positive definite.
\end{proof}

\subsection{Proof of Proposition \ref{Prop-St2}} 
Let  $\alpha < 0$, $p\ge 4$  and consider $S_{\alpha}(t):= S_{\alpha,0}(t)=  e^{it(\alpha \partial_x^4 + \partial_y^2)}$. We are going to prove that 
\begin{equation}
\big\|S_{\alpha}(t)\phi\big\|_{L^p([0,1]_t \times \T^2)} \lesssim  N^{1-\frac7{2p}+\epsilon}\|\phi\|_{{L^2(\T^2)}}
\end{equation}
holds for any $\phi\in L^2(\T^2)$ with  ${\textrm{supp}} (\widehat{\phi}) \subset [-\sqrt{N/2}, \sqrt{N/2}]\times [-N, N]$.

The idea is to apply  the $\ell^2$ decoupling result from Theorem A stated in Section \ref{sec-2} on the surface  $\mathbb{S}_{\alpha}$  in $\R^3$ defined by \eqref{h-surface}.

For any $\phi \in L^2(\T^2)$, one has
\begin{equation}\label{Fourier-group}
e^{it(\alpha \partial_x^4 + \partial_y^2)} \phi(x,y) = \sum_{(\xi_1, \xi_2) \in \Z^2} \widehat{\phi}(\xi_1, \xi_2) e^{2 \pi i (x \xi_1 + y \xi_2 + t (- \alpha \xi_1^4+ \xi_2^2) )}.
\end{equation}

Let $N\in \N$ and consider $\xi=(\xi_1, \xi_2)\in \Z^2$ with $-\sqrt{N/2}\leq \xi_1 \leq \sqrt{N/2}$ and $-N\leq \xi_2 \leq N$. Now, we define a collection of $1/N$ separated points $\eta=(\eta_1, \eta_2)$ in the following way:
\begin{align}
&\eta_1 = \frac{1}{\sqrt{2N}}\xi_1 , \; \;\;   \eta_2 = \frac{\xi_2}{2 N}\label{etas}\\
\intertext{and also we define}
& a_{\eta} := \widehat{\phi}(\xi)\label{a-eta}.
\end{align}

From \eqref{etas} one gets
\begin{align*}
&\xi_1 = \sqrt{2N}\eta_1, \quad  \xi_2 = 2N \eta_2\\
\intertext{and}
&t (- \alpha \xi_1^4+ \xi_2^2) = t (2 N)^2 (-\alpha \eta_1^4 + \eta_2^2).
\end{align*}
So, setting
\[
x = \frac{1}{\sqrt{2N}}\,x', \quad y = \frac{1}{2N} y'\quad \mbox{and}\quad t = \frac{t'}{(2N)^2},
\]
we have 
\[
x \xi_1 + y \xi_2 = \eta_1 x' + \eta_2 y'.
\]

Note that, in this setting, one has
$$	
x'\in \big[0, \sqrt{2N}\big],\quad y'\in [0, 2N]\quad \text{and}\quad  t'\in [0, 4N^2].
$$

In what follows we consider $p\ge 4$ and we recall that $\widehat{\phi} \in \ell^2(\Z^2)$ with support in the rectangle $[-\sqrt{N/2}, \sqrt{N/2}]\times [-N, N]$. The above change of variables shows that
\begin{equation}\label{eq5.x1}
\begin{split}
\int_{[0,1]_t\times\T^2  } |&e^{it(\alpha \partial_x^4 + \partial_y^2)} \phi|^p dx dy dt\\
&  \lesssim   \frac{1}{N^{7/2}} \int_{[0, 4N^2]\times [0,  \sqrt{2N}] \times [0, 2N] }
\Big|\sum_{(\eta_1, \eta_2)} a_{\eta} e^{2 \pi i (x' \eta_1 + y' \eta_2 + t' (-\alpha\eta_1^4 + \eta_2^2))} \Big|^p dt' dx' dy'. 
\end{split}
\end{equation}

Now, using  the periodicity in the  $x'$ and $y'$ variables,  the estimate in \eqref{eq5.x1} yields
\begin{equation}\label{eq5.x2}
\begin{split}
\big\|S_{\alpha}&(t)\phi\big\|^p_{L^p([0,1]_t \times \T^2)}\\
& \lesssim  \frac{1}{N^{11/2}} \int_{[0, 4N^2]\times [0, \sqrt{2}N^{3/2}] \times [0, 2N^2] }
\Big|  \sum_{(\eta_1, \eta_2) } \widehat{\phi}(\eta_1, \eta_2) e^{2 \pi i (x' \eta_1 + y' \eta_2 + t' (-\alpha\eta_1^4 + \eta_2^2))} \Big|^p dx' dy' dt' \\
& \lesssim  \frac{N^{1/2}}{N^6} \int_{[0, 4N^2]\times [0, 2N^2] \times [0, 2N^2] }
\Big|  \sum_{(\eta_1, \eta_2) } \widehat{\phi}(\eta_1, \eta_2) e^{2 \pi i (x' \eta_1 + y'\eta_2 + t' (-\alpha\eta_1^4 + \eta_2^2))} \Big|^p dx' dy' dt'.
\end{split}
\end{equation}

So, setting $P_N:=[0, 4N^2]\times [0, 2N^2] \times [0, 2N^2]$ we have from \eqref{eq5.x2} that
\begin{equation}\label{eq5.x3}
\begin{split}
\big\|S_{\alpha}(t)\phi&\big\|_{L^p([0,1]_t \times \T^2)}\\
&\lesssim N^{\frac1{2p}}\left(\frac{1}{N^6} \int_{P_N}
\Big|  \sum_{(\eta_1, \eta_2) } \widehat{\phi}(\eta_1, \eta_2) e^{2 \pi i (x' \eta_1 + y' \eta_2 + t' (-\alpha\eta_1^4 + \eta_2^2))} \Big|^p dx' dy' dt'\right)^{1/p}.	
\end{split}
\end{equation}

\medskip 
From Lemma \ref{lemma-CS} the compact surface $\mathbb{S}_{\alpha}$ given by \eqref{h-surface} possesses positive definite second fundamental form.  Therefore, applying the decoupling Theorem A with $n=3$,  $\delta^{1/2} = \frac1N$ (i.e., $\delta = N^{-2}$) and $p\ge \frac{2(3+1)}{3-1}=4$, one gets from \eqref{eq5.x3} that 
\begin{equation}\label{eq5.x4}
\big\|S_{\alpha}(t)\phi\big\|_{L^p([0,1]_t \times \T^2)}\lesssim  N^{\frac1{2p}} N^{-2\big(\frac2p-\frac12 -\epsilon\big)}\|\phi\|_{{L^2(\T^2)}}.	
\end{equation}

In particular, taking  $p=4$ in \eqref{eq5.x4} we arrive at 
\begin{equation*}
	\big\|S_{\alpha}(t)\phi\big\|_{L^4([0,1]_t \times \T^2)}\lesssim  N^{\frac1{8}+2\epsilon}\|\phi\|_{{L^2(\T^2)}},	
\end{equation*}
and the proof is finished.

\section{Concluding Remarks }\label{sec-6}
We considered the non-isotropic nonlinear Schr\"odinger equation \eqref{model} that appears in the fiber arrays posed on two dimensional product domains, viz., cylindrical and purely periodic domains. New well-posedness results for the associated IVP are obtained. Let us note some important points:

\medskip 
\begin{enumerate}[(a)]
\item In the case when the IVP was posed on the $\T\times \R$ with $\alpha <0$, we are able to explore the good behaviour of the associated symbol to obtain Strichartz estimate similar to the one for the classical NLS equation in 2d thereby getting global well-posedness (w.p.) result for small data  in $H^s(\T\times \R)$ whenever $s\geq 0$. This result is proved to be sharp  by showing that the application data-solution is not uniformly continuous.
\begin{table}[h!]
	\centering
	\begin{tabular}{|c|@{}c@{}|}\hline
\cellcolor{olive!15} $\T\times\R$
&
	\begin{tabular}{|c|c|c|}
		\hline
		\cellcolor{olive!15} $\boldsymbol{\varepsilon}$  &\cellcolor{olive!15}  $\boldsymbol{\alpha}$ &\cellcolor{olive!15}  \textbf{local w.p.\, Sobolev regularity} $s$ \\
		\hline
		 0,\, 1 & $<0$ & $s\ge 0$ with small data for $s=0$\\
		\hline
		 0,\, 1 & $>0$ &  open question\\
		\hline
	\end{tabular}
	\tabularnewline\hline
	\end{tabular}
\end{table}

\item When the problem is posed on the cylindrical domain $\R\times \T$, the symbol behaves very badly preventing us to obtain estimates analogous to Lemma~\ref{crucial-lem}.  In this domain (in general case $\alpha \neq 0$) we worked on the anisotropic Sobolev spaces  exploiting the  one dimensional Strichartz estimate in the $x$-variable proved in \cite{GC2009} to obtain a new Strichartz-type estimate for a family of equations involving Fourier mode of  periodic $y$-variable.  Finally, summing up the resulting estimates we obtained a full Strichartz-type estimate and used it to get local well-posedness result in the anisotropic Sobolev spaces   $\mathcal{H}_{x,y}^{s_1,s_2}(\R\times\T)$ for $s_1\geq 0$ and $s_2>\frac12$. We do not know if this result is sharp. So, this is an interesting open question.
\begin{table}[h!]
	\centering
	\begin{tabular}{|c|@{}c@{}|}\hline
\cellcolor{olive!15} $\R\times \T$
&
	\begin{tabular}{|c|c|c|}
		\hline
		\cellcolor{olive!15} $\boldsymbol{\varepsilon}$  &\cellcolor{olive!15}  $\boldsymbol{\alpha}$ &\cellcolor{olive!15}  \textbf{local w.p.\,for anisotropic regularity} $(s_1,s_2)$ \\
		\hline
		0,\, 1 & $\alpha \neq 0$ & $s_1\geq 0$ and $s_2>1/2$.\\
		\hline
		0,\, 1 & $\alpha \neq 0$ &  open question for $s_1\geq 0$ and $0\le s_2 \le 1/2.$ \\
		\hline
	\end{tabular}
	\tabularnewline\hline
	\end{tabular}
\end{table}

\item In the purely periodic case, i.e., for the IVP posed on $\T^2$, we used the $\ell^2$ decoupling theory developed by Bourgain and Demeter \cite{BD2015}. To apply this theory, homogeneity in each variable of the linear part plays a crucial role in obtaining a compact hypersurface with a positive definite second fundamental form. The lack of homogeneity in the $x$-variable created an additional difficulty in obtaining such surface. However, for  a particular case of the IVP \eqref{model} with $\varepsilon =0$ we could ensure homogeneity in both $x$ and $y$ variables separately. With this consideration we  obtained an  $L^4-L^2$-Strichartz estimate with $\frac18+\epsilon$ derivative loss thereby getting local well-posedness result in $H^s(\T^2)$ for $s>\frac14$. Having this information at hand, one may ask if a similar result can be obtained for the original IVP \eqref{model}. It is a very interesting and challenging problem on which the authors are planning to work in the future project. 
\begin{table}[h!]
	\centering
	\begin{tabular}{|c|@{}c@{}|}\hline
\cellcolor{olive!15} $\T^2$
&
	\begin{tabular}{|c|c|c|}
		\hline
		\cellcolor{olive!15} $\boldsymbol{\varepsilon}$  &\cellcolor{olive!15}  $\boldsymbol{\alpha}$ &\cellcolor{olive!15}  \textbf{local w.p.\, Sobolev regularity} $ s$ \\
		\hline
		0 & $<0$ & $ s>1/4$ \\
		\hline
		0 & $<0$ & open question for $s\le 1/4$\\
		\hline
		1 & $<0$ & open question\\
		\hline
		0, 1 & $>0$ & open question\\
		\hline
	\end{tabular}
	\tabularnewline\hline
	\end{tabular}
\end{table}
\end{enumerate}

 Finally, we record the following IVP
\begin{equation}\label{model-2}
\begin{cases}
i \partial_t u +  \varepsilon\partial_x^2u + \partial_y^2u+\alpha_1  \partial_x^{4}u +\alpha_2 \partial_x^{6}u = \pm |u|^2 u, & t\in \R,\; (x, y) \in \cD,\medskip \\
u(0, x, y) = \phi(x,y),   &(x, y) \in \cD,
\end{cases}
\end{equation}
where $u$ is a complex valued function, $\varepsilon \in \{0, 1\}$ and  $\alpha_1, \alpha_2 \in \R$.  For $\varepsilon=1$, this model was proposed in \cite{FP-1, FP-2} in the context of fiber arrays.

\medskip
We believe that the technique developed in this article for the IVP \eqref{model} can be adapted for the IVP \eqref{model-2} posed on $\T\times \R$ considering appropriate signs of $\alpha_1$ and $\alpha_2$ so that one can solve the corresponding cubic polynomial as in \eqref{biquadratic-sol}, and usual adaptation when posed on $\R\times\T$ . However, for the IVP \eqref{model-2} posed on $\T^2$, due to non-homogeneity in the linear part, it is not clear if one can use the same technique as that for the IVP \eqref{model} even with $\varepsilon =0$. So, it is an interesting open problem.

\medskip
Biharmonic NLS equation  is another interesting problem to consider on cylinders and purely periodic domain $\T^2$. We believe that the recent theory developed in \cite{GMO} by Guth-Maldague-Oh could be applied to handle the problem in $\T^2$. These are the problems we are working in our ongoing project.
\medskip

\section*{Acknowledgments}
The authors thank the anonymous referee for valuable comments that helped improve the presentation of this work. The authors are also grateful to Dr. Yuzhao Wang for his insightful comments, which helped identify the correct regularity threshold for the local well-posedness result in the purely periodic case.
The first and second authors extend their thanks to the Department of Mathematics, UNICAMP, Campinas for kind hospitality and excellent working conditions, which greatly facilitated the initiation of this work.  The second author acknowledges supports from FAPESP, Brazil (Grants \# 2024/20513-7 \& 2023/06416-6) and FAPEMIG, Brazil (project \# RED- 00133-21). The third author acknowledges supports from FAPESP, Brazil (Grants \# 2024/10613-4 \& 2023/06416-6) and CNPq, Brazil (\# 307790/2020-7), and thanks the University of Birmingham, UK, for its kind hospitality, where part of this work was carried out.\\

\medskip


\noindent
{\bf Conflict of interest statement.} 
On behalf of all authors, the corresponding author states that there is no conflict of interest.\\

\noindent 
{\bf Data availability statement.} 
Datasets generated during the current study are available from the corresponding author on request.

\medskip

\bibliographystyle{plain}

\end{document}